\documentclass{amsart}
\usepackage{amssymb,amsmath}

\usepackage{float}
\usepackage{subfig,epsfig}
\usepackage{graphicx}

\newtheorem{theorem}{Theorem}[section]
\newtheorem{lemma}[theorem]{Lemma}

\theoremstyle{definition}

\theoremstyle{remark}
\newtheorem{remark}[theorem]{Remark}

\numberwithin{equation}{section}

\begin{document}
\renewcommand{\theequation}{\thesection.\arabic{equation}}
\makeatletter
\@namedef{subjclassname@2020}{%
	\textup{2020} Mathematics Subject Classification}
\makeatother

	\author{Fengjiang Li}
	\address[Fengjiang Li]{ School of Mathematical Sciences\\ East China Normal University\\
		Shanghai, 200241}
	\email{lianyisky@163.com}

	%\date{\today}

\title[VSS with harmonic curvature]
{Vacuum static spaces with harmonic curvature}

\thanks{}

\keywords{vacuum static space, harmonic curvature, Codazzi tensor, $D$-fat}

\subjclass[2020]{Primary 53C21; Secondary 53C25, 83C20}

\begin{abstract}
In this paper, we classify $n$-dimensional ($n\geq 5$) vacuum static spaces with harmonic curvature, thus extending the $4$-dimensional work by Kim-Shin \cite{KS}. As a consequence, we provide new counterexamples to the Fischer-Marsden conjecture on compact vacuum static spaces.
\end{abstract}

\maketitle

%\vfill
%\pagebreak
\setcounter{section}{0}
\setcounter{equation}{0}

\section{Introduction}
An $n$-dimensional Riemannian manifold  $(M^n, g)$ 
 is said to be a {\it vacuum static space} 
 if there exists a nonzero smooth (lapse) function $f$ on $M$ such that
 %is a complete Riemannian manifold  $(M^n, g,f)$ that admits a smooth non-trivial solution $f$ (i.e., $f\neq0$) to
\begin{equation} \label{static}
Hess f = f(Rc -\frac{R}{n-1} g).
\end{equation}
Vacuum static metrics arise naturally in the study of static space-times in general
relativity.
Actually, let $(M,g)$ be an $n$-dimensional Riemannian manifold and $f$ be a nonzero smooth positive function on $M$.
It can be shown that the Lorentzian manifold $(\mathbb R \times M, -f^2dt^2+g )$ satisfies the Einstein equation for the energy momentum tensor $T$ of a perfect fluid if and only if $(M,g,f)$ is a vacuum static space satisfying \eqref{static}, (cf. \cite{Kobayashi-Obata,I,Co,Wald84}).

It is very interesting to notice that the vacuum static equation \eqref{static} is also considered by Fischer and Marsden \cite{FM} in their study of the surjectivity of scalar curvature function on the space of Riemannian metrics, originally derived from the linearization of the scalar curvature equation (cf. \cite{Bo,FM,KS}).
Moreover, Bourguignon \cite{Bo} and Fischer-Marsden \cite{FM} independently proved that a complete vacuum static space has constant scalar curvature, which is necessarily non-negative if further $M$ is compact.
Based on this fact, Fischer-Marsden \cite{FM} made the following conjecture.   

\medskip
\noindent \textbf{Fischer-Marsden Conjecture.}
\emph{Any compact vacuum static space is an Einstein space.}

\medskip
If it is true, by Obata's theorem \cite{Ob}, such a space must be a standard sphere or a Ricci flat space.
However, it turns out the conjecture is not true.
The first counterexample was provided by Kobayashi \cite{Ko} and Lafontaine \cite{La} independently. They proved that compact locally conformally flat vacuum static spaces are ${\mathbb S}^n$, ${\mathbb S}^{1} \times {\mathbb S}^{n-1}$ and certain warped product $ {\mathbb S}^{1} \times_h {\mathbb S}^{n-1}$.	
Moreover, Kobayashi-Obata \cite{Kobayashi-Obata} showed that 
if a complete vacuum static metric $(M^n, g,f)$ is locally conformally flat, then it is isometric to a warped product $I\times_h N^{n-1}$ of an open interval $I \subset {\mathbb R}$ with $N$ of constant sectional curvature. 
Furthermore, 
Kobayashi \cite{Ko} constructed five important examples of such warped products and ultimately gave the classification of complete locally conformally flat vacuum static spaces.
Later, Qing-Yuan \cite{QY} classified complete Bach-flat vacuum static spaces with compact level sets. 
Inspired by the work of Cao-Chen \cite{CC,CC2}, they defined a covariant $3$-tensor $D$ for vacuum static spaces. In fact, what they obtained is essentially the classification result of D-flat vacuum static spaces; see Section \ref{prevss} for more details. For more results, we refer to \cite{Be,Co,CEM,CS,S,YN} and the references therein.

\smallskip
Recently, Kim-Shin \cite{KS} studied the $4$-dimensional vacuum static space of harmonic curvature and obtained a local description of the metric and potential function.
Their method of proof was motivated by Kim's work \cite{Kim} 
on the classification of $4$-dimensional gradient Ricci solitons with harmonic Weyl curvature.
In a very recent paper \cite{L}, the author has succeeded in extending the work of Kim \cite{Kim} and classified $n$-dimensional gradient Ricci solitons with harmonic Weyl curvature for all $n\ge 5$.

In this paper, in a way similar to that of gradient Ricci solitons \cite{L}, we will study $n$-dimensional ($n\geq 5$) vacuum static space with harmonic curvature satisfying \eqref{static}, and extend the work of Kim-Shin \cite{KS} to all dimensions $n\ge 5$. Our main result is the following

\begin{theorem} \label{complete}	
Let $(M^n, g,f )$ be an $n$-dimensional, $n\geq5$, complete vacuum static space with harmonic curvature satisfying \eqref{static}.
Then it is one of the following types:
	
\smallskip
{\rm (i)} $(M,g)$ is $D$-flat.  Consequently, by Qing-Yuan \cite{QY}, $(M,g)$ is either Einstein or isometric to one of the spaces in Examples 1-5 as described in Section \ref{prevss}.

\smallskip
{\rm (ii)}
$(M,g)$ is isometric to a quotient of $\mathbb{S}^2\left(\frac{R}{2(n-1)}\right) \times   N^{n-2} $ with $R>0$, 
where $\left(N^{n-2}, {g_2}\right)$ is Einstein with positive Einstein constant $\frac{R}{n-1}$.
$f=c_1\cos\left( \sqrt{\frac{R}{2(n-1)}}s \right)$ for some constant $c_1\neq0$, 
where $s$ is the distance on $\mathbb{S}^2(\frac{R}{2(n-1)})$ from a point.
		
\smallskip
{\rm (iii)} $(M,g)$ is isometric  to a quotient of 
$ \mathbb{H}^2\left(\frac{R}{2(n-1)}\right) \times  N^{n-2}$ with $R<0$, 
where  
$\left(N^{n-2}, {g_2}\right)$ is Einstein with negative Einstein constant $\frac{R}{n-1}$.
$f=c_2\cosh\left( \sqrt{-\frac{R}{2(n-1)}}s \right)$ for some constant $c_2\neq0$, 
where $s$ is the distance function on $ \mathbb{H}^2\left(\frac{R}{2(n-1)}\right)$ from a point.

\smallskip
{\rm (iv)} $(M,g)$ is isometric to a quotient of the Riemannian product $(W^{r}\times N^{n-r}, g=\bar{g}+g_2)$, where $3 \leq r\leq n-1$, 
$( W^{r},\bar{g}=ds^2+h^2(s)\tilde{g})$ is an $r$-dimensional $D$-flat vacuum static space
and $ \left( N^{n-r},g_2\right) $ is an Einstein manifold of Einstein constant $\frac{R}{n-1}$. $f=ch'$ for some constant $c$.
Particularly, $R=0$ for $r= n-1$.
\end{theorem}

Finally, we pick up compact spaces in Theorem \ref{complete}.

\begin{theorem} \label{cpt}
Let $(M^n, g,f )$ be an $n$-dimensional, $n\geq5$, compact vacuum static space with harmonic curvature satisfying \eqref{static}.
Then it is one of the following types:

\smallskip
{\rm (i)} $(M,g)$ is $D$-flat. Hence, $(M,g)$ is either isometric to the Euclidean sphere $\mathbb S^n$, or a quotient of $(\mathbb{S}^1,ds^2)\times(N^{n-1},g_0)$, 
or a quotient of a product torus $(\mathbb S^1\times N^{n-1}, ds^2 +h^2(s)g_0)$,
where $( N^{n-1},g_0)$ is Einstein.

\smallskip
{\rm (ii)}
$(M,g)$ is isometric to a quotient of $\mathbb{S}^2\left(\frac{R}{2(n-1)}\right) \times   N^{n-2} $ with $R>0$, 
where $\left(N^{n-2}, {g_2}\right)$ is Einstein with positive Einstein constant $\frac{R}{n-1}$.
$f=c_1\cos\left( \sqrt{\frac{R}{2(n-1)}}s \right)$ for some constant $c_1\neq0$, 
where $s$ is the distance on $\mathbb{S}^2(\frac{R}{2(n-1)})$ from a point.
	
\smallskip
{\rm (iii)} $(M,g)$ is isometric to a quotient of the Riemannian product $(W^{r}\times N^{n-r}, g=\bar{g}+g_2)$, where $3 \leq r\leq n-1$,
$( W^{r},\bar{g}=ds^2+h^2(s)\tilde{g})$ is an $r$-dimensional $D$-flat vacuum static space
and $ \left( N^{n-r},g_2\right) $ is an Einstein manifold of the Einstein constant $\frac{R}{n-1}$. $f=ch'$ for some constant $c$.
Particularly, $R=0$ for $r= n-1$.
\end{theorem}

\begin{remark}
It is worth noting that types (ii) and (iii) will provide new counterexamples, which are not $D$-flat, to the Fischer-Marsden conjecture. Moreover, Examples in type (iii) are the Riemannian products of $D$-flat vacuum static spaces and Einstein manifolds, which are similar to the rigid gradient Ricci solitons that are the Riemannian products of Gaussian solitons and Einstein manifolds. Therefore, we see that the $D$-flatness property plays an important role in vacuum static spaces.
\end{remark}

This paper is organized as follows. 
In Section \ref{prevss}, we give some formulae and notations for Riemannian manifolds and vacuum static spaces by using the method of moving frames.
In Section \ref{vss3}, we derive the integrability conditions (ODEs) for a vacuum static space with harmonic curvature
and show that, locally, the metric is a multiply warped product.
In Section \ref{vss4}-\ref{vss5}, in order to complete the proof of Theorem \ref{complete}, we divide our discussion into three cases according to the numbers and multiplicities of distinct
Ricci-eigenvalues, excluding the one with respect to the gradient vector of the lapse function.

\section{Preliminaries}\label{prevss}
In this section, we first recall some formulae and notations 
for Riemannian manifolds by using the method of moving frames.
Then we give some facts on vacuum static spaces.

\subsection{Some notations for Riemannian manifolds.}
Let $M^{n}(n \geq 3)$ be an
$n$-dimensional Riemannian manifold,
$E_{1}, \cdots, E_{n}$ be a local orthonormal frame
fields on $M^{n}$, and $\omega_{1}, \cdots, \omega_{n}$
be their dual 1-forms. In this paper we make the following
conventions on the range of indices:
\[
1\leq i,j,k,\cdots\leq n
\]
and agree that repeated indices are summed over the respective ranges. 
Then we can write the structure equations of $M^{n}$ as follows:
\begin{equation}\label{2.1}
d\omega_{i}=\omega_{j}\wedge\omega_{ji}
\quad {\rm and }\quad \omega_{ij}+\omega_{ji}=0;
\end{equation}
\begin{equation}\label{2.2}
-\frac{1}{2}R_{ijkl}\omega_{k}\wedge\omega_{l}=
d\omega_{ij}-\omega_{ik}\wedge\omega_{kj}
\quad {\rm and }\quad R_{ijkl}=-R_{jikl},
\end{equation}
where $d$ is the exterior differential operator on $M$,
$\omega_{ij}$ is the Levi-Civita connection form
and $R_{ijkl}$ is the Riemannian curvature tensor of $M$.
It is known that the Riemannian curvature tensor satisfies
the following identities:
\begin{equation}\label{2.3}
R_{ijkl}=-R_{ijlk}, \quad R_{ijkl}=R_{klij}
\quad {\rm and }\quad R_{ijkl}+R_{iklj}+R_{iljk}=0.
\end{equation}
The Ricci tensor $R_{ij}$ and scalar curvature $R$ are defined respectively by
\begin{equation}\label{2.4}
R_{ij}:=\sum\limits_{k}R_{ikjk} \quad {\rm and }\quad R=\sum\limits_{i}R_{ii}.
\end{equation}
Let $f$ be a smooth function on $M^{n}$, we define the
covariant derivatives $f_{i}$, $f_{i,j}$ and $f_{i,jk}$ as follows:
\begin{equation}\label{2.5}
f_{i}\omega_{i}:=df,\quad  f_{i,j}\omega_{j}:=df_{i}+f_{j}\omega_{ji},
\end{equation}
and
\begin{equation}\label{2.6}
f_{i,jk}\omega_{k}:=df_{i,j}+f_{k,j}\omega_{ki}+f_{i,k}\omega_{kj}.
\end{equation}
We know that
\begin{equation}\label{2.7}
f_{i,j}=f_{j,i} \quad {\rm and }\quad f_{i,jk}-f_{i,kj}=f_{l}R_{lijk}.
\end{equation}
The gradient, Hessian and Laplacian of $f$ are defined by the following formulae:
\begin{equation}\label{2.8}
\nabla f:= f_{i}E_{i}, \quad  Hess(f):=f_{i,j}\omega_{i}\otimes \omega_{j} \quad {\rm and }\quad 
\Delta f:=\sum\limits_{i}f_{i,i}.
\end{equation}
The covariant derivatives of tensors $R_{ij}$ and $R_{ijkl}$
are defined by the following formulae:
\begin{equation}\label{2.9}
R_{ij,k}\omega_{k}:=dR_{ij}+R_{kj}\omega_{ki}+R_{ik}\omega_{kj}
\end{equation}
and
\begin{equation}\label{2.10}
R_{ijkl,m}\omega_{m}:=dR_{ijkl}+R_{mjkl}\omega_{mi}+
R_{imkl}\omega_{mj}+R_{ijml}\omega_{mk}+R_{ijkm}\omega_{ml}.
\end{equation}
By exterior differentiation of \eqref{2.2}, one can get the second Bianchi identity
\begin{equation}\label{2.11}
R_{ijkl,m}+R_{ijlm,k}+R_{ijmk,l}=0.
\end{equation}
From \eqref{2.4}, \eqref{2.10} and \eqref{2.11}, we have
\begin{equation}\label{2.12}
R_{ij,k}-R_{ik,j}=-\sum\limits_{l}R_{lijk,l},
\end{equation}
and so
\begin{equation}\label{2.13}
\sum\limits_{j}R_{ji,j}=\frac{1}{2}R_{i}.
\end{equation}

We define the Schouten tensor as
$A=A_{ij}\omega_{i}\otimes \omega_{j},$
where
\begin{equation}\label{2.14}
A_{ij}:=R_{ij}-\frac{1}{2(n-1)}R \delta_{ij},
\end{equation}
then $A_{ij}=A_{ji}$. The tensor
\begin{equation}\label{2.15}
W_{ijkl}:=R_{ijkl}-\frac{1}{n-2}
(A_{ik}\delta_{jl}+A_{jl}\delta_{ik}-
A_{il}\delta_{jk}-A_{jk}\delta_{il})
\end{equation}
is called the Weyl conformal curvature tensor.% which does not change under the conformal transformation of the metric. Moreover, as it can be easily seen by the formula above, $W$ is totally trace-free. 
In dimension three, $W$ is identically zero on every Riemannian manifold, whereas,
when $n\geq 4$, the vanishing of the Weyl tensor is equivalent to the locally
conformal flatness of $(M^n,g)$. We also recall that in dimension $n=3$, 
$(M, g)$ is locally conformally flat iff the Cotton tensor $C$, defined as follows, vanishes
\begin{equation}\label{2.16}
C_{ijk} := A_{ij,k} - A_{ik,j}.
\end{equation}
We recall that, for $n\geq 4$, using the second Bianchi identity 
the Cotton tensor can also be defined as one of the possible divergences of the Weyl tensor:
\begin{equation}\label{2.17}
-\frac{n-2}{n-3}\sum\limits_{l} W_{lijk,l}= C_{ijk}.
\end{equation}

On any $n$-dimensional manifold $(M, g)$ $(n\geq 4)$,
in what follows a relevant role will be played by the Bach tensor, 
first introduced in general relativity by Bach \cite{bac} in early 1920s'. By definition,
\begin{equation}\label{2.18}
B_{ij}: = \frac{1}{n-3}W_{ikjl, kl} + \frac{1}{n-2}R_{kl}W_{ikjl}
\end{equation}
and by \eqref{2.17}, we have an equivalent expression of the Bach tensor:
\begin{equation}\label{2.19}
B_{ij} =\frac{1}{n-2} \left( C_{ijk, k}+R_{kl}W_{ikjl}\right).
\end{equation}

\subsection{The $D$-tensor and $D$-flat vacuum static spaces}
We will recall the $D$-tensor defined in \cite{QY} for vacuum static spaces and the classification result for $D$-flat vacuum static spaces by Qing-Yuan \cite{QY}.

The covariant 3-tensor $D_{ijk}$ (see \cite{QY}) is defined by
\begin{equation}
\begin{aligned}
\label{D}
D_{ijk}=&\frac{n-1}{n-2}\left(R_{ik}f_j-R_{ij}f_k\right)
+\frac{R}{n-2}\left(f_k\delta_{ij}-f_j\delta_{ik}\right) \\
&+\frac{1}{n-2}f_l\left(R_{lj}\delta_{ik}-R_{lk}\delta_{ij}\right),
\end{aligned}
\end{equation}
which is the analog of the $D$-tensor defined by Cao-Chen \cite{CC,CC2} for Ricci solitons.
First of all, we have the following lemma.
 
\begin{lemma}[{\bf Qing-Yuan \cite{QY}}]
	\label{lemDCW}
Let $(M^n,g,f)$ be a vacuum static space satisfying \eqref{static}. Then the following formulae hold:
\begin{equation}\label{DCW}
fC_{ijk}=f_lW_{lijk}+D_{ijk}.
\end{equation}
\begin{equation}\label{DB}
(n-2)fB_{ij}=D_{ijk,k}
-f_k\left( C_{ijk} + \frac{n-3}{n-2}C_{jik}\right) .	
\end{equation}
\end{lemma}

Next, we recall the following classification of $D$-flat vacuum static spaces, which was obtained by Qing-Yuan \cite{QY} (even though it was not explicitly stated).

\begin{theorem} {\bf(Qing-Yuan \cite{QY})}\label{d-flat}
	 Suppose that $(M^n, g,f)$ is an $n$-dimensional $D$-flat vacuum static space satisfying \eqref{static}. Then the metric $g$ is a locally warped product,
	\[
	g= ds^2 +  h^2(s) \tilde{g}
	\]
	for a positive function $h$, where the Riemannian metric $\tilde{g}$ is Einstein with the Einstein constant $(n-2)k$.
	$(M,g)$ is either Einstein or isometric to one of the spaces in Examples 1-5, as follows.
	
	Furthermore, $h$ satisfies the following equation
	\begin{equation} \label{h}
	h^{''} + \frac{R}{n(n-1)}h = c_0h^{-(n-1)}
	\end{equation}
	for a constant $c_0$ and 
	\begin{equation} \label{hk}
	(h^{'})^2 + \frac{2c_0}{n-2}h^{-(n-2)} + \frac{R}{n(n-1)}h^2 =k.
	\end{equation}
	The non-constant lapse fucntion $f$ satisfies
	\begin{equation}\label{fh}
	h'f'-fh''=0.
	\end{equation}
\end{theorem}

To compare with Kobayashi's work \cite{Ko} on locally conformally flat vacuum static spaces, we note that the only difference is that $\tilde{g}$ is an Einstein metric in the $D$-flat case, while it is of constant sectional curvature in the locally conformally flat case. Therefore, as noted in \cite{QY}, one only needs to replace the constant curvature space factor of the warped products in \cite{Ko} with corresponding Einstein manifold factor to obtain the following five $D$-flat examples.

\smallskip
{\bf Example 1}:
Let $N(k)$ be an $(n-1)$-dimensional connected complete Einstein space
with the Einstein constant $(n-2)k$ . 
On a Riemannian product $\mathbb R\times N^{n-1}(k)$, $k\neq 0$,
\begin{equation}
f(s)=\left\{\begin{array}
{ll}c_{1}\sin\sqrt{(n-2)k} s+c_{2}\cos\sqrt{(n-2)k}s+x, &{\rm if} ~ k>0,\\
c_{1}\sinh\sqrt{-(n-2)k} s++c_{2}\cosh\sqrt{-(n-2)k} s+x, &{\rm if} ~ k<0,
\end{array}\right.
\end{equation}
where $c_{1}$ and $c_{2}$ are constants. 

\smallskip
{\bf Example 2}: Compact quotients of Example 1 with $k>0$, with an explicit isometry group $\Gamma_{l, \phi}$ generated by $(s, x) \to (s+2l\pi/\sqrt{(n-2)k}, \,\phi(x))$, where $l$ is any natural number and $\phi \in Isom (N(k))$. 

\smallskip
{\bf Example 3}. 
A warped product ${\mathbb R} \times_h N^{n-1}(1)$, where $h$ is a periodic solution to \eqref{h} with $a>0, R>0, k>k_0=\frac R {(n-1)(n-2)} (\frac {n(n-1)a} {R})^{2/n}$ (as in Proposition 2.6 of Kobayashi \cite{Ko}), and non-trivial $f$ satisfying $f(s)=ch'$ for some constant $c$ . 

\smallskip
{\bf Example 4}. Since $f$ and $h$ in Example 3 have a common period, we obtain
compact spaces with non-trivial solutions to \eqref{static} in a similar way as Example 2.
These spaces were first found by Ejiri \cite{Ejiri} as counterexamples for a different problem.
 
\smallskip
{\bf Example 5}. 
 A warped product ${\mathbb R} \times_h N^{n-1}(k)$, where $h, k$ are as in Proposition 2.5 of Kobayashi \cite{Ko} (i.e, $h,\, k$ satisfies one of the conditions (IV.1)-(IV.4)), and non-trivial $f(s)=ch'$ for some constant $c$.
 
\begin{remark}
	Suppose that $(M^n, g)$ is a warped product $\left(I\times\, _h N^{n-1},\,ds^2 +  h^2(s) \tilde{g}\right)$, where $h$ is not constant and $\left(N^{n-1},~ h^2(s) \tilde{g}\right)$
	is Einstein with the Einstein constant $(n-2)k$.
	Then there exists a smooth function $f$ depending only on $s$ satisfying \eqref{static}
	if and only if $(M^n, g)$ is of constant scalar curvature, as well as functions $f$ and $h$ satisfy \eqref{h} and \eqref{fh}. 	
\end{remark}

\smallskip
One can now pick up compact spaces in Theorem \ref{d-flat}.
\begin{theorem}{\bf(Qing-Yuan \cite{QY})}\label{Dcpt}
	Let $(M^n, g,f)$ be a compact $D$-flat vacuum static space satisfying \eqref{static}.
	Then $M$ is isometric to the Euclidean sphere $\mathbb S^n$, or the quotient of 
	$(\mathbb{S}^1,ds^2)\times(N^{n-1},g_0)$, or the quotient of a product torus $(\mathbb S^1\times N^{n-1}, ds^2 +h^2(s)g_0)$, where $( N^{n-1},g_0)$ is an Einstein manifold.
\end{theorem}

\section{The basic local structure for $n$-dimensional vacuum static spaces with harmonic curvature}\label{vss3}
In the following sections, in a similar way to that of Ricci soliton \cite{L}, we will give the (local) classification of vacuum static space with harmonic curvature.
In this section, the goal is to derive the integrability conditions (ODEs) 
for a vacuum static space with harmonic curvature, and then show that, locally, the metric is a multiply warped product.
	
\smallskip
 Let $(M^n, g, f)$, $n\geq4$, be an $n$-dimensional  vacuum static space with harmonic curvature satisfying \eqref{static}. First of all, we recall the next lemma (see also Lemma 2.5 in \cite{KS}).
\begin{lemma} \label{lemma3.1}
In some neighborhood $U$ of each point in $\{ \nabla f \neq 0  \}$, 
we choose an orthonormal frame field
$\{E_1= \frac{\nabla f}{|\nabla f| }, E_2, \cdots,  E_n \}$ 
with the dual frame field $\{\omega_1= \frac{d f}{|\nabla f| }, \omega_2, \cdots, \omega_n\}$.
Then we have following properties
		
\smallskip
{\rm (i)} $E_1= \frac{\nabla f }{|\nabla f | }$ is an eigenvector field of the Ricci tensor.
		
\smallskip
{\rm (ii)} The 1-form $\omega_1= \frac{d f}{|\nabla f| }$ is closed. 
So the distribution $V={\rm Span}\{E_2, \cdots , E_n\}$ is integrable from the Frobenius theorem. 
We denote by $L$ and $N$ the integrable curve of the vector field $E_1$ 
and the integrable submanifold of $V$ respectively.
Then it follows locally $M=L\times N$ 
and there exists a local coordinates $(s, x_2, \cdots , x_n)$ of $M$ such that 
$ ds= \frac{d f}{|\nabla f|}$,  $E_1 = \nabla s$,
$V={\rm Span} \{\frac{\partial}{\partial x_2}, \cdots , \frac{\partial}{\partial x_n}\} $
and
$g=ds^2+\sum_{a}\omega_a^2$.
		
\smallskip
{\rm (iii)}
$f$ and $R_{11}$ can be considered as functions of the single variable $s$.
\end{lemma}
	
\begin{proof}
Firstly, it follows from \eqref{2.12} that the Ricci tensor of a Riemannian metric with harmonic curvature is a Codazzi tensor, i.e. $R_{ij,k}=R_{ik,j}$.
Moreover, for any $n$-dimensional manifold $(M^n, g, f)$ with harmonic curvature satisfying \eqref{static},
it holds that
\begin{equation}\label{harcon}
f_l R_{lijk}=f_k\left(R_{ij}-\frac{R}{n-1}\delta_{ij} \right)
-f_j\left(R_{ik}-\frac{R}{n-1}\delta_{ik} \right)
\end{equation}
where the Ricci identity \eqref{2.7} was used. Noting that 
$f_1=|\nabla f|\neq 0$, $f_a=0$ for $2\leq a\leq n$, setting $i=j=1$, and $k=a$ in \eqref{harcon}, we have 
$f_1 R_{1a}=0$.
Hence $R_{1a}=0$, i.e., $E_1= \frac{\nabla f }{|\nabla f | }$ is an eigenvector field of the Ricci curvature.
		
Next, $R_{1a}=0$ leads to  
$f_{1,a}
= f(R_{1a} -\frac{R}{n-1} \delta_{1a})=0$, and then 
\[
\left(  \left| \nabla f \right|^ 2\right)_a = 2f_1f_{1,a}=0.
\]
Therefore 
\[
d\omega_1 = d\left( \frac{d f}{|\nabla f|}\right) 
=-\frac{1}{2 |\nabla f|^{\frac{3}{2}}} d \left( |\nabla f|^{2}\right)  \wedge df
=0,
\]
that is $\omega_1= \frac{d f}{|\nabla f| }$ is closed. 

Note that $f$ can be considered as a function of the single variable $s$.
From \eqref{2.5}, it follows that $f_{1,1}=f''$.
Since $f\neq0$ on an open subset, we have
\[
R_{11}=f^{-1}\left( f_{1,1} +\frac{R}{n-1}\right) 
\]
depending only on $s$. We have completed the proof of this lemma.
\end{proof}

For any $n$-dimensional ($ n\geq 3$) Riemannian manifold with harmonic curvature, the Ricci tensor satisfies the Codazzi equation. As described on Codazzi tensors by Derdzi\'{n}ski \cite{De,De82},
for any point $x$ in $M$, let $E_{\rm Ric}(x)$ be the number of distinct eigenvalues of ${\rm Ric}_x$,
and set 
\[
M_{\rm Ric} = \{  x \in M \ | \ E_{\rm Ric} {\rm \ is \ constant \ in \ a \ neighborhood \ of \ } x \}.
\] 
Then $M_{\rm Ric}$ is an open dense subset of $M$ and in each connected component of $M_{\rm Ric}$, the eigenvalues are well-defined and differentiable functions.
A Riemannian manifold with harmonic curvature is real analytic in harmonic coordinates \cite{De}, i.e., $f$ is real analytic (in harmonic coordinates).
Then if $f$ is not constant, $\{\nabla f \neq 0\}$ is open and dense in $M$.  
			
Therefore, in some neighborhood $U$ of each point in $M_{\rm Ric} \cap \{ \nabla f \neq 0  \}$, the number of distinct Ricci-eigenvalues is constant, and we assume there are $m$ distinct Ricci-eigenvalues of multiplicities
$r_{1}, r_{2}, \cdots, r_{m}$, respectively, excluding the one with corresponding to the eigenvector $\frac{\nabla f }{|\nabla f | }$.
Here $1+r_{1}+r_{2}+ \cdots+ r_{m}=n $.
Therefore, we can choose an orthonormal frame field
$\{E_1= \frac{\nabla f}{|\nabla f| }, E_2, \cdots,  E_n\}$,
with the dual $\{\omega_1= \frac{d f}{|\nabla f| }, \omega_2, \cdots, \omega_n\}$ 
such that 
\begin{equation}\label{3.2}
R_{ij}=\lambda_i \delta_{ij},
\end{equation}
\[
\lambda_{2}=\cdots=\lambda_{r_1+1}, \quad
\lambda_{r_1+2}=\cdots=\lambda_{r_1+r_2+1},\quad
\cdots  \quad 
\lambda_{r_1+r_2+\cdots+r_{m-1}+2}=\cdots=\lambda_{n}, 
\]
and $\lambda_{2},\, \lambda_{r_1+2},\,\cdots,\,	\lambda_{r_1+r_2+\cdots+r_{m-1}+2}$
are distinct. Next, we need the following lemma to prove that all Ricci-eigenfunctions $ \lambda_i $ $(i=1, \cdots ,n)$ depend only on the local variable $s$. 

\begin{lemma} \label{lemma3.2}
Let $(M^n, g, f)$, $(n\geq 4)$, be an $n$-dimensional vacuum static space with harmonic curvature satisfying \eqref{static}. For the above local frame field $\{ E_i \}$ in $M_{\rm Ric} \cap \{ \nabla f \neq 0 \}$, we have
	\begin{equation}\label{3.3}
f''=f\lambda_1-\frac{R}{n-1},
	\end{equation}
	\begin{equation}\label{3.4}
	\omega_{1a}= \xi_a \omega_{a}
	\end{equation}
	and 
	\begin{equation}\label{3.5}
	R_{aa,1}
	=\left(\lambda_1-\lambda_a \right)\xi_a,
	\end{equation}
	where 
	\begin{equation} \label{3.6}
	\xi_a:= \frac{1}{ |\nabla f|} \left( f\lambda_a-\frac{R}{n-1} \right).
	\end{equation} 
\end{lemma}

\begin{proof}	
   It follows from \eqref{2.5} that for $2\leq a \leq n$,  
	\[
	f_{1,1}=f'' \quad  {\rm and} \quad  f'\omega_{1a}=f_{a,j}\omega_{j}.
	\]
  Substituting $f_{i,j}=\left( f\lambda_i-\frac{R}{n-1} \right)\delta_{ij}$ into the above equations, we immediately have \eqref{3.3} and \eqref{3.4}.
  Directly calculation by using of \eqref{2.9} gives us 
	\[
	R_{11,1}=\lambda_1'  \quad {\rm and} \quad 
	R_{1a,a}=\left(\lambda_1-\lambda_a \right)\xi_a.
	\]
	Hence harmonic condition $R_{aa,1}=R_{1a,a}$ yields \eqref{3.7},
	and we have completed the proof of this lemma.
\end{proof}

Making use of the above lemmas, it is easy to prove the following lemma,
which was proved by Kim-Shin \cite{KS} for the four dimensional case.
\begin{lemma}%{\bf (Kim-Shin \cite{KS}) } 
	\label{lemma3.3}
	Let $(M^n, g, f)$, $(n\geq 4)$, be an $n$-dimensional vacuum static space with harmonic curvature satisfying \eqref{static}.
	For the above local frame field $\{ E_i \}$ in $M_{\rm Ric} \cap \{ \nabla f \neq 0 \}$,
	the Ricci eigenfunctions $ \lambda_i $ $(i=1, \cdots ,n)$
	depend only on the local variable $s$, so do the functions $\xi_a$ for $2 \leq a\leq n$.
\end{lemma}

Subsequently, we will obtain the local structure of the metric for an $n$-dimensional vacuum static space with harmonic curvature.
First, we denote $[a]=\{b| \lambda_{b}=\lambda_{a}~ {\rm and } ~ b \neq1\}$ for $2\leq a\leq n$
and make the following conventions on the range of indices:
\[
2 \leq a, b, c \cdots\leq n  \quad {\rm and} \quad  2 \leq \alpha, \beta,  \cdots\leq n
\]
where $[a]=[b]$, $[\alpha]=[\beta]$ and $[a]\neq[\alpha]$.

\begin{lemma} \label{lemma3.3-}
	Let $(M^n, g, f)$, $(n\geq 4)$, be an $n$-dimensional vacuum static space  with harmonic curvature satisfying \eqref{static}.
	For the above local frame field $\{ E_i \}$ in $M_{\rm Ric} \cap \{ \nabla f \neq 0 \}$, we see that
	\begin{equation}\label{3.7}
	\omega_{a\alpha}=0,
	\end{equation}
	\begin{equation}\label{3.8}
	R_{1a1b}=-\left( \xi'_a+\xi^2_a\right) \delta_{ab}
	\end{equation}
	and 
	\begin{equation}\label{3.9}
	R_{a\alpha b \beta}=-\xi_a\xi_\alpha \delta_{ab}\delta_{\alpha\beta}.
	\end{equation}
\end{lemma}

\begin{proof}	
	Using the same method to that of Lemma 3.4 in \cite{L}, we can easily prove this lemma.
\end{proof}

At last, we are ready to derive the integrability conditions and prove the local structure of the metric as a multiply warped product.
 
\begin{theorem} \label{mulwar}
	Let $(M^n, g, f)$, $(n\geq 4)$, be an $n$-dimensional vacuum static space with harmonic curvature satisfying \eqref{static}.
	For each point $p \in M_{\rm Ric} \cap \{ \nabla f \neq 0  \}$, there exists a neighborhood $U$ of $p$ 
	such that 
	\[
	U = L\times\,  _{h_1}L_1\times \cdots \times\, _{h_l}L_l \times\, _{h_{l+1}}N_{l+1}\times\cdots \times\, _{h_m}N_{m}
	\]
	is a multiply warped product furnished with the metric 
	\begin{equation} \label{3.10}
	g= ds^2 + h^2_1(s)  dt_1^2 + \cdots 
	+ h^2_l(s) dt_l^2+ h^2_{l+1}(s) \tilde{g}_{l+1}+ \cdots 
	+h^2_{m}(s) \tilde{g}_{m},
	\end{equation}
	where $h_j(s)$ are smooth positive functions for $1\leq j \leq m$,
	$dim ~L_\nu$=1 for $1\leq \nu \leq l$, and 
	$(N_\mu, \tilde{g}_{\mu})$ is an $r_\mu$-dimensional Einstein manifold 
	of the Einstein constant $(r_\mu-1)k_\mu$ for $l+1\leq \mu \leq m$.
	
	Moreover, 
	let $\lambda_i$ $(i=1,\cdots, n)$ be the Ricci-eigenvalues, $\lambda_1$ being the Ricci-eigenvalue with respect to the gradient vector $\nabla f$,
	and functions $\xi_a$ be given by \eqref{3.6}.	
	Then the following integrability conditions hold
	\begin{equation}\label{3.11}
	\xi'_a+\xi^2_a=\lambda_a-\frac{R}{n-1},
	\end{equation}
	\begin{equation}\label{3.12}
	\lambda'_a-\left(\lambda_1-\lambda_a \right)\xi_a=0,
	\end{equation}
	\begin{equation}\label{3.13}
		\lambda_{1}=f^{-1}\left( f''+\frac{fR}{n-1}\right)=-\sum_{i=2}^n\left(\xi'_i+\xi^2_i\right),
	\end{equation}
	and
	\begin{equation}\label{3.14}
	\begin{aligned}
	\lambda_{a}&=f^{-1}\left( f'\xi_a+\frac{fR}{n-1}\right)\\
	&=- \xi'_a-\xi_a \sum^{n}_{i=2} \xi_i +( r-1)\frac{k}{h^2},
	\end{aligned}
	\end{equation}
	Here $2\leq a\leq n$ and $\xi_a=h'/h$;  in \eqref{3.14}, when $a=2, \dots, l+1$,  then 
	$r=1$ and $h=h_{a-1}$;
	when $a\in \left[ l+r_{l+1}+\cdots+r_{\mu-1} +2\right]$ ,  then 
	$r=r_\mu$, $h=h_{\mu} $ and $k=k_{\mu} $.
\end{theorem}
\begin{proof}
	In a similar way to that of Theorem 3.4 in \cite{L}, we can easily prove this lemma.
     Here, we only prove the partially integrability conditions, different from that of Ricci solitons.
	
	Making use of \eqref{harcon} again, we see
	\[ 
	R_{1a1a}=-R_{aa}+\frac{R}{n-1}
	\]
	which shows that \eqref{3.11} holds with \eqref{3.8}.
	Then in combination with \eqref{3.3}, we immediately get \eqref{3.13}. Putting $R_{aa,1}=\lambda'_1$ into \eqref{3.5} implies harmonic condition \eqref{3.12}.
	The rest proof can be referred to the paper \cite{L}.
\end{proof}

\section{The local structure of the case other than two distinct Ricci-eigenfunctions}\label{vss4}
In this section and the next section, 
we will prove the local classification of $n$-dimensional vacuum static spaces with harmonic curvature according to how many distinct Ricci-eigenvalues and their multiplicities. 
To avoid repetition, unless stated otherwise, the Ricci-eigenvalues mentioned in the following discussions do not include $\lambda_1$, which is the eigenvalue with respect to the gradient vector $\nabla f$ of the lapse function. Also, we denote the other Ricci-eigenvalues by $\lambda_a$, $a=2,\cdots, n$.  

\medskip

Firstly, we treat the case that all Ricci-eigenfunctions $\{\lambda_a\}$, $2\le a\le  n$, are equal.  Type {\rm (i)} of Theorem \ref{complete} will come from this case.
\begin{theorem} \label{same}
	Suppose that $(M^n, g,f)$ is an $n$-dimensional vacuum static space with harmonic curvature satisfying \eqref{static}. If all Ricci-eigenfunctions $\{\lambda_a\}$, $2\le a\le  n$, are equal.	Then locally the metric $g$ is a warped product of the form
	\[
	g= ds^2 +  h^2(s) \tilde{g}
	\]
	with a certain positive warping function $h$,
	where the Riemannian metric $\tilde{g}$ is Einstein. 
	In particular, the $D$-tensor of  $(M^n, g,f)$ vanishes (and so does the Bach tensor). 	
\end{theorem}

\begin{proof}
	The first part of this theorem has been obtained by Theorem \ref{mulwar}.		
	For the second part, since the harmonicity implies that the Cotton tensor $C_{ijk}=0$,
	with the relationships \eqref{DCW} and \eqref{DB},
	it is easy to check that the vacuum static space is $D$-flat and Bach-flat of constant scalar curvature.
	The detail also can be referred to \cite{L}.
\end{proof}

\begin{remark}
	  In combination with Theorem \ref{d-flat}, it follows that for a vacuum static space $(M^n, g,f)$ satisfying \eqref{static}, $(M^n, g)$ is a locally warped product, described by the hypothesis of Theorem \ref{same}, if and only if the $D$-tensor vanishes. 
	In addition, from the proof of Theorem \ref{same}, we have that $D=0$ implies $B=0$, but the reverse is not true. 
\end{remark}

Next, we shall study the case
when $\lambda_2,~\lambda_3,~\cdots,~\lambda_n$ are at least three mutually different. As we shall see below, 
it turns out that this case cannot occur.

\medskip	
First of all, we analyze the integrability conditions in Theorem \ref{mulwar}.
Assume that $\lambda_a$ and $\lambda_\alpha$ are mutually different of multiplicities $r_1$ and $r_2$, i.e.,
\[
\lambda_{a}:=\lambda_{2}= \cdots= \lambda_{r_1+1} \quad {\rm and} \quad
\lambda_{\alpha}:=\lambda_{r_1+2}= \cdots=\lambda_{r_1+r_2+1};
\]
Here we make the following conventions on the range of indices:
\[
2 \leq a, b,  \cdots\leq ({r_1+1}); \quad ({r_1+2}) \leq \alpha, \beta,  \cdots\leq ({r_1+r_2+1});\quad 
\] 
Denote $\xi_a:=X$ and $ \xi_\alpha:=Y$, from Section \ref{vss3}, and then they satisfy the following integrability conditions:
\begin{equation}\label{4.1}
X'+X^2-\lambda_a=Y'+Y^2-\lambda_\alpha=\xi'_i+\xi^2_i-\lambda_i=-\frac{R}{n-1},
\end{equation}
\begin{equation}\label{4.2}
\lambda_{1}=f^{-1}\left( f''+\frac{R}{n-1}f\right)=-\sum_{i=2}^n\left(\xi'_i+\xi^2_i\right),
\end{equation}
\begin{equation}\label{4.3}
\lambda_{a}=f^{-1}\left( f'X+\frac{R}{n-1}f\right)
=-( X'+X^2)-X\sum_{i=2}^n\xi_i+(r_1-1)\frac{k_1}{h^2_1}+X^2,
\end{equation}
\begin{equation}\label{4.4}
\lambda_{\alpha}=f^{-1}\left( f'Y+\frac{R}{n-1}f\right)
=-( Y'+Y^2)-Y\sum_{i=2}^n\xi_i+(r_2-1)\frac{k_2}{h^2_2}+Y^2
\end{equation}
and 
\begin{equation}\label{4.5}
\lambda'_a-\left(\lambda_1-\lambda_a \right)X=\lambda'_\alpha-\left(\lambda_1-\lambda_\alpha \right)Y=0.
\end{equation}

\medskip
By using the above basic facts, it is easy to get the following equations:
\begin{lemma}\label{lemma4.1}	
	Let $(M^n, g, f)$, $(n\geq 4)$, be an $n$-dimensional vacuum static space with harmonic curvature.
	In some neighborhood $U$ of $p\in M_{\rm Ric} \cap \{ \nabla f \neq 0  \}$, 
	let $\lambda_a$ and $\lambda_\alpha$ are mutually different Ricci-eigenvalues with multiplicities $r_1$ and $r_2$.
	Then the following identities hold:

\begin{equation}\label{4.6}
\left\{2f'+f\left[ \sum_{i=2}^n\xi_i-(X+Y) \right] \right\} (X-Y)=f\left[ (r_1-1)\frac{k_1}{h^2_1}-(r_2-1)\frac{k_2}{h^2_2} \right]
\end{equation}
	and
	\begin{equation}\label{4.7}
	-f^{-1}f'\sum_{i=2}^n\xi_i+\sum_{i=2}^n \xi'_i+2\sum_{i=2}^n \xi^2_i=(X+Y)\sum_{i=2}^n\xi_i.
	\end{equation}
\end{lemma}
\begin{proof}
	Subtracting \eqref{4.3} from \eqref{4.4} gives us 
		\begin{equation}\label{4.8}
	\lambda_{a}-\lambda_\alpha=f^{-1} f'(X-Y)
	\end{equation}
    and 
	\[
	\lambda_{a}-\lambda_\alpha=-(X-Y)'-(X-Y)\sum_{i=2}^n\xi_i+\left[(r_1-1)\frac{k_1}{h^2_1}
	- (r_2-1)\frac{k_2}{h^2_2}\right].
	\]
	From the first equality of \eqref{4.1},
	 it follows that $\lambda_{a}-\lambda_\alpha=(X-Y)'+(X^2-Y^2)$.
	 Then respectively comparing with \eqref{4.8} and the above equation, we have
	 \begin{equation}\label{4.9}
	 X'-Y'=(X-Y)[f^{-1}f'-(X+Y)]
	 \end{equation}
  and 
	\[
		2(\lambda_{a}-\lambda_\alpha)=-( X-Y)\left[ \sum_{i=2}^n\xi_i-(X+Y) \right] +\left[(r_1-1)\frac{k_1}{h^2_1}
	- (r_2-1)\frac{k_2}{h^2_2}\right].
	\]
  Putting \eqref{4.8} into the above equation, we can immediately see \eqref{4.6}.	
   Then differentiating the above equation yields	
	\begin{equation*}
	\begin{aligned} 
	2(\lambda_{a}-\lambda_\alpha)'
	=&-( X-Y)'\sum_{i=2}^n\xi_i-( X-Y)\sum_{i=2}^n\xi'_i\\
	&-2\left[\left( (r_1-1)\frac{k_1}{h^2_1}-X'\right)X
	- \left( (r_2-1)\frac{k_2}{h^2_2}-Y'\right)Y\right],
	\end{aligned}
	\end{equation*}
	 where $X=h'_1/h_1$ and $Y=h'_2/h_2$ were used.
	On the other hand, it follows from  \eqref{4.5} that
	\begin{equation*}
	\begin{aligned}
	& \lambda'_a -\lambda'_{\alpha}=\lambda_{1}(X-Y)-\lambda_{a}X+\lambda_{\alpha}Y \\
	=&-(X-Y)\sum_{i=2}^n\left(\xi'_i+\xi^2_i\right)+(X^2-Y^2)\sum_{i=2}^n\xi_i\\
	&-\left[\left( (r_1-1)\frac{k_1}{h^2_1}-X'\right)X
	- \left( (r_2-1)\frac{k_2}{h^2_2}-Y'\right)Y\right],
	\end{aligned}
	\end{equation*}
	where \eqref{4.1} and \eqref{4.2} were used.
	Comparing with the above two equations leads to
	\[
		-[(X-Y)'+2(X^2-Y^2)]\sum_{i=2}^n\xi_i+(X-Y)\sum_{i=2}^n \xi'_i+2(X-Y)\sum_{i=2}^n\xi^2_i=0.
	\]	
 Putting $(X-Y)'+2(X^2-Y^2)=(X-Y)(f^{-1}f'+X+Y)$, rewritten by \eqref{4.9}, into the above equation, 
   we see \eqref{4.7} since $X$ and $Y$ are different.
 We have completed the proof of this lemma.
\end{proof}	

Next, we will apply equations \eqref{4.1}, \eqref{4.2} and \eqref{4.7} to deal with the case when $\lambda_2,~\lambda_3,~\cdots,~\lambda_n$ are at least three mutually different values.		
\begin{theorem}\label{thm4.2}
	Let $(M^n, g, f)$, $(n\geq 4)$, be an $n$-dimensional vacuum static space with harmonic curvature.
	Then in some neighborhood $U$ of $p \in M_{\rm Ric} \cap \{ \nabla f \neq 0  \}$, 
	the Ricci-eigenvalues $\lambda_2,~\lambda_3,~\cdots,~\lambda_n$ cannot be more than two distinct values. 
\end{theorem}

\begin{proof}
	If not, we assume that $\lambda_2,~\lambda_3,~\cdots,~\lambda_n$ are at least three mutually different values,
	and denote by $\lambda_a$, $\lambda_\alpha$ and $\lambda_p$ with multiplicities $r_1$, $r_2$ and $r_3$.
	For convenience, we also denote $\xi_a:=X$, $\xi_\alpha:=Y$ and  $\xi_{p}:=Z$. By \eqref{4.7}, 
	with the assumption of Lemma \ref{lemma4.1}, we see that
	\[
	(X+Y)\sum_{i=2}^n\xi_i=(X+Z)\sum_{i=2}^n\xi_i=(Y+Z)\sum_{i=2}^n\xi_i.
	\]
	If $\sum_{i=2}^n\xi_i\neq0$, the above equation implies that $X=Y=Z$ and this contradicts the hypothesis.
	If $\sum_{i=2}^n\xi_i=0$, \eqref{4.7} implies $\sum_{i=2}^n \xi^2_i=0$. Hence $\xi_i=0$ for each $i=2,\,\dots,\,n$, which is a contradiction.
	We have completed the proof of Theorem \ref{thm4.2}.
\end{proof}

\section{The local structure of the case with exactly two distinct Ricci-eigenfunctions}\label{vss5}

In this section we begin to study the case when there are exactly two distinct Ricci values in the eigenvalues 
$\lambda_2, \cdots, \lambda_n$.	
Types {\rm (ii)}, {\rm (iii)} and {\rm (iv)} of Theorem \ref{complete} come from this section.
	
\smallskip
First of all, we need the following lemma to prepare for the local structure of the case.

\begin{lemma}\label{lemma5.1} 
	Let $(M^n, g, f)$, $(n\geq 4)$, be an $n$-dimensional vacuum static space with harmonic curvature.
	Assume in some neighborhood $U$ of $p \in M_{\rm Ric} \cap \{ \nabla f \neq 0 \}$, there are exactly two distinct values in the Ricci-eigenvalues $\lambda_2, \cdots, \lambda_n$, 
	denoted by $\lambda_a$ and $\lambda_\alpha$ of multiplicities $r_1$ and $r_2:=n-r_1-1$.
	Then one of functions $X$ and $Y$ vanishes. 
\end{lemma}
\begin{proof}
In this case, equation \eqref{4.6} can be rewritten as 
\begin{equation*}
\begin{aligned}
&-(X'-Y')(r_1X+r_2Y)+(X-Y)(r_1X'+r_2Y')\\
&+2(r_1X^2+r_2Y^2)(X-Y)-2(X^2-Y^2)(r_1X+r_2Y)=0.
\end{aligned}
\end{equation*}	
Directly simplifying the above equation yields that
\[
X'Y-XY'+2XY(X-Y)=0.
\]
Meanwhile, the integrability conditions \eqref{3.11} and \eqref{3.14} imply that $f\xi'_i=f'\xi_i-f\xi^2_i$ for each $i=2,~\dots,~,n$. Then putting $fX'=f'X-fX^2$ and $fY'=f'Y-fY^2$ into the above equation, we see $XY=0$,
which means that one of functions $X$ and $Y$ vanishes. 
We have completed the proof of this lemma.
\end{proof}
		
From now on, without loss of generality, we assume that $X\neq0$ and $Y=0$. Meanwhile, $Y=0$ implies that $h_2$ is a constant.

Next, we will discuss two cases according to 
whether one of the two Ricci-eigenfunctions is of single multiplicity.

\subsection{One of the multiplicities of two Ricci-eigenfunctions $\lambda_a$ and $\lambda_\alpha$ is single}
In this subsection, we will study the case 
that one of the multiplicities of two Ricci-eigenfunctions $\lambda_a$ and $\lambda_\alpha$ is single 
in some neighborhood $U$ of $p$ in $M_{\rm Ric} \cap \{ \nabla f \neq 0  \}$.
Types {\rm (ii)}, {\rm (iii)} and {\rm (iv)} for $r=n-1$ of Theorem \ref{complete} come from this subsection.

\medskip
\noindent{\bf Subcase I.}   \quad  $r_1=1$ and $r_2=n-2\geq2$ 
	
For this subcase, the following classification results form types {\rm (ii)} and {\rm (iii)} of Theorem \ref{complete}.
	
 We denote $h_1=h$ and then locally the metric is given by $g = ds^2 + h^2(s) dt^2 + {g_2}$.
Firstly, we claim $R\neq0$. If not, then $R=0$. 
\eqref{4.1} means $\lambda_a=X'+X^2 $, while \eqref{4.3} means $\lambda_a=-(X'+X^2) $. Hence we see
$\lambda_a=X'+X^2=0$ and $\lambda_\alpha=0$ by \eqref{4.4}. 
This is a contradiction.
	
Next, making use of \eqref{4.1} and \eqref{4.3} again, 
we see $X'+X^2 =-\frac{R}{2(n-1)}\neq0$ and 
$ \frac{h''}{h} =-\frac{R}{2(n-1)}\neq0$.
On the other hand, \eqref{4.4} gives us $(r_2-1)\frac{k_2}{h^2_2}=\frac{R}{n-1}$,
which implies that ${g_2}=h^2_2\tilde{g_2}$ 
is an Einstein metric with the Einstein constant $\frac{R}{n-1}$.

If $R> 0$, setting $r_0 =\sqrt{\frac{R}{2(n-1)}} $,   
then, $h = C_1 \sin( r_0s ) $ for some constant $C_1\neq 0$ and $X=  r_0 \cot( r_0s ) $.
From $f'X=f(X'+X^2)$, it follows that $f= c_1 \cos( r_0s )$ for some nonzero constant $c_1$.
Thus $g= ds^2 + \sin^2( r_0s )dt^2+{g_2}$ by absorbing a constant into $dt^2$ and using 
$\lambda_\alpha=\frac{k_2}{h^2_2} = \frac{R}{n-1}$.
Here $s$ is the distance on $\mathbb{S}^2\left(\frac{R}{2(n-1)}\right)$ from a point.

If $R < 0$, we set $r_0 =\sqrt{-\frac{R}{2(n-1)}}$. 
One can argue similarly as above, and get  $g= ds^2 +\sinh^2( r_0s )dt^2 +{g}_2$  
 and $f= c_2 \cosh( r_0s  )$ for some nonzero constant $c_2$.

\medskip
Consequently, we have a conclusion as follows.
\begin{theorem} \label{twoone1}
Let $(M^n, g, f)$, $(n\geq 4)$, be an $n$-dimensional vacuum static space with harmonic curvature.
Aussme locally the metric is given by $g = ds^2 + h^2(s) dt^2 + {g_2}$,
and ${g_2}$ is an Einstein metric.
Then $g$ is of nonzero scalar curvature $R$;

\smallskip
{\rm (i)} when $R>0$, $(M^n, g)$ is locally isometric  to the Riemannian product $\mathbb{S}^2\left(\frac{R}{2(n-1)}\right) \times N^{n-2}_2 $ and $ f=c_1\cos\left( \sqrt{\frac{R}{2(n-1)}}s \right)$ for some nonzero constant $c_1$, where $s$ is the distance on $\mathbb{S}^2\left(\frac{R}{2(n-1)}\right)$ from a point and the Einstein manifold $\left(N^{n-2}_2, {g_2}\right)$ is of positive Einstein constant $\frac{R}{n-1}$.
	
\smallskip
{\rm (ii)} when $R<0$, $(M^n, g)$ is locally isometric to the Riemannian product  
$\mathbb{H}^2\left(\frac{R}{2(n-1)}\right) \times   N^{n-2}_2 $, 
where the Einstein manifold $\left(N^{n-2}_2, {g_2}\right)$ is of negative Einstein constant $\frac{R}{n-1}$,
and $f=c_2\cosh\left( \sqrt{-\frac{R}{2(n-1)}}s \right) $ for some nonzero constant $c_2$.
\end{theorem}

\medskip
\noindent{\bf Subcase II.}  \quad  $r_2=1$ and $r_1=n-2\geq2$.

For this subcase, we will obtain the classification result,
which forms type {\rm (iv)} of Theorem \ref{complete} for $r=n-2$ .
 
We denote $h_1=h$ and then locally the metric is $g = ds^2 + h^2(s) g_1 + dt^2$.
Firstly, from \eqref{4.4} we see $R=0$. Meanwhile,
\[
f^{'} \frac{h^{'}}{h}=f'X=f(X'+X^2)= f\frac{h^{''}}{h}
\]
implies $ch^{'} =f$ for a constant $c \neq 0$.
By \eqref{4.2}, we have
\[
	 ch^{'''}=f^{''}=-(n-2)f(X'+X^2)=-(n-2)ch'\frac{h^{''}}{h}
\]
and
$h^{(n-2)} h^{''} =c_0$
 for some constant $c_0\neq0$. In fact, if $c_0=0$, $h^{''}=0$
and $ X'+X^2=h''/h=0$. It follows from \eqref{4.3} that
$\lambda_a=X'+X^2=0$ and $\lambda_\alpha=0$. 
This is a contradiction.
	
By \eqref{4.1} and \eqref{4.3}, we see
\[
	0=2(X'+X^2)+(n-3)(X^2-\frac{k}{h_1^2})
	= 2\frac{h^{''}}{h}+(n-3)\frac{h'^2-k}{h^2},
\]
 which implies $
	 2hh''+(n-3)(h'^2-k)=0$.
Combining with $h^{(n-2)} h^{''} =c_0$, we obtain that
\[
	 h'^{2}+\frac{2c_0}{n-3}h^{-(n-3)}=k.
\]
Hence, in comparing with \eqref{h}, it is easy to see that $(W^{n-1}=\mathbb R\times N_1,\,\bar g=ds^2+h^2(s)g_1)$
is an $(n-1)$-dimensional vacuum static space with vanishing $D$ tensor, 
and the scalar curvature $\bar{R}=R-\lambda_{n}=0$.

\medskip
Consequently, we have the following conclusion in this subcase.
\begin{theorem}	\label{twoone2}
	Let $\left(M^n, g, f \right) $, $n\geq4$, be a vacuum static space with harmonic curvature satisfying \eqref{static}. Assume locally the metric is given by $g=ds^2+h^2(s)g_1+dt^2$, and ${g_1}$ is an Einstein metric.
Then $(M, g) $ is locally isometric to a domain in $( W^{n-1}   \times \mathbb{R}^1, g_W + dt^2  )$,
where $( W^{n-1}, g_W)$ is an $(n-1)$-dimensional $D$-flat vacuum static space of zero scalar curvature and $f = ch'$.
\end{theorem}

\subsection{The multiplicities of two Ricci-eigenfunctions $\lambda_a$ and $\lambda_\alpha$ are more than one}	
In this subsection, we will study the multiplicities of two Ricci eigenfunctions 
$\lambda_a$ and $\lambda_\alpha$ are more than one
in some neighborhood $U$ of a point $p$ in $M_{Ric} \cap \{ \nabla f \neq 0  \}$.
We have the local classification results, 
which form type {\rm (iv)} of Theorem \ref{complete} for $3\leq r\leq n-2$.
	
For this case $r_1,~r_2\geq2$,  $Y=0$ and $X\neq0$.
 From $f'\xi_i=f(\xi'_i+\xi^2_i)$, we have
\[
	 f^{'} \frac{h^{'}}{h} =f'X=f(X'+X^2)= f\frac{h^{''}}{h},
\]
 and $ch^{'} =f$ for some constant $c \neq 0$.
By using of equations \eqref{4.1} and \eqref{4.3}, we see
\[
	 ch^{'''}=f^{''}=-f\left( r_1(X'+X^2)+\frac{R}{n-1}\right) 
	 =-ch'\left( r_1\frac{h^{''}}{h}+\frac{R}{n-1}\right).
\]
for some constant $c_0$, and
\begin{equation} \label{5.1}
	 h^{''}+\frac{R}{(n-1)(r_1+1)}h=c_0h^{-r_1}.
\end{equation}
From \eqref{4.1} and \eqref{4.3}, we see
\[
	 0=2(X'+X^2)+(r_1-1)(X^2-\frac{k}{h_1^2})+\frac{R}{n-1}
	 = 2\frac{h^{''}}{h}+(r_1-1)\frac{h'^2-k}{h^2}+\frac{R}{n-1},
\]
which implies
\[
 2h''+(r_1-1)(h'^2-k)+\frac{R}{n-1}=0.
\]
Putting \eqref{5.1} into the above equation, we get 
\begin{equation} \label{5.2}
	 (h^{'})^2 + \frac{2c_0}{r_1-1}h^{-(r_1-1)} + \frac{R}{(n-1)(r_1+1)}h^2 =k.
\end{equation}
	 
Now, we consider the manifold $W^{r_1+1}=\mathbb{R}^{1}\times N^{r_1}_1$
with $\bar{g}=ds^2 + h^2{g}_1$. 
The scalar curvature is given by  $\bar{R}=R-r_2\lambda_\alpha=\frac{r_1}{n-1}R$.
From \eqref{5.1} and \eqref{5.2}, we have respectively,
\[
h^{''}+\frac{\bar{R}}{r_1(r_1+1)}h=c_0h^{-r_1}
\]
and
\[
(h^{'})^2 + \frac{2c_0}{r_1-1}h^{-(r_1-1)} + \frac{\bar{R}}{r_1(r_1+1)}h^2 =k.
\]
Moreover, $f=ch^{'}$.
It is now easy to see that $( W^{r_1+1},\bar{g})$ is $D$-flat vacuum static space explained in Section \ref{prevss}.
	
	\medskip 
Consequently, we have the following conclusion in this subcase.
\begin{theorem}\label{twodis}
Let $(M^n, g, f)$, $(n\geq 4)$, be an $n$-dimensional vacuum static space with harmonic curvature satisfying \eqref{static}.
Assume locally the metric is given by 
 $g=ds^2+h^2(s){g}_1+g_2$, where ${g}_1$ and $g_2$ are both Einstein metrics.
Then $(M^n, g)$ is locally isometric to a domain in $( W^{r_1+1}\times N^{n-1-r_1}_2, g=\bar{g}+g_2 )$,
where $2\leq r_1\leq n-3$, $( W^{r_1+1},\bar{g})$ is an $(r_1+1)$-dimensional $D$-flat vacuum static space and $ \left( N^{n-1-r_1}_2,g_2\right) $ is an Einstein manifold of the Einstein constant $\frac{R}{n-1}$.
\end{theorem}

\smallskip

Finally, from Theorems \ref{same}, \ref{d-flat}, \ref{twoone1}, \ref{twoone2} and \ref{twodis}, it follows that among types {\rm (i)}-{\rm (iv)} Theorem \ref{complete}, each type is different from the other three types.
Therefore, Theorem \ref{complete} holds from continuity argument of complete Riemannian metrics.
 
\bigskip
\noindent {\bf Acknowledgments.}
This work was completed while the author was visiting Lehigh University from August, 2019 to August, 2020. 
She would like to thank her advisor Professor Huai-Dong Cao  
for his invaluable guidance, constant encouragement and support.
She is grateful to her advisors Professor Yu Zheng and Professor Zhen Guo for their constant encouragement and support.
She also would like to thank Junming Xie, Jiangtao Yu, and other members of the geometry group at Lehigh for their interest, helpful discussions, and suggestions during the preparation of this paper. 
She also would like to thank the China Scholarship Council (No: 201906140158)
for the financial support, and the Department of Mathematics at Lehigh University for hospitality and for providing a great environment for research.


\begin{thebibliography}{99}
\bibitem{bac} R. Bach. {\em Zur Weylschen Relativi{t\:at}stheorie und der Weylschen Erweiterung des Krummungstensorbegriffs}, Math. Z., 9 (1921), 110--135.

\bibitem{Be} A.L. Besse: {\em Einstein manifolds}. 
Ergebnisse der Mathematik, 3 Folge, Band 10, Springer-Verlag, 1987.

\bibitem{Bo} J.P. Bourguignon, {\em Une stratifcation de l'espace des structures riemanniennes},
Compositio Math. 30 (1975), 1-41.

\bibitem{CC} H.-D. Cao and Q. Chen, {\em On locally conformally flat gradient steady Ricci solitons},
Trans. Amer. Math. Soc. 364 (2012), 2377-2391.

\bibitem{CC2} H.-D. Cao and Q. Chen,  {\em  On Bach-flat gradient shrinking Ricci solitons}, 
Duke Math. J., 162 (2013), no. 6, 1003-1204.

\bibitem{CL} H.-D. Cao and F.J. Li, {\em Besse conjecture and critical spaces with harmonic curvature},
preprint, 2020.

\bibitem{CY} H.-D. Cao and J.T. Yu, {\em on complete gradient steady Ricci solitons with vanishes $D$-tensor},
preprint, 2020.

\bibitem{Co} J. Corvino: {\em Scalar curvature deformation and a gluing construction for the Einstein constraint equations}, 
Comm. Math. Phys., 214 (2000), 137-189.

\bibitem{CEM} J. Corvino, M. Eichmair and P. Miao, {\em  Deformation of scalar curvature and volume},
 Math. Ann., 357 (2013), no. 2, 551-584.

\bibitem{CS} J. Corvino and R.M. Schoen,  {\em  On the Asymptotics for the Vacuum Einstein Constraint Equations},
J. Differential Geom., 73 (2006), no. 2, 185-217.

\bibitem{De} A. Derdzi\'{n}ski, {\em Classification of Certain Compact Riemannian Manifolds with Harmonic Curvature and Non-parallel Ricci Tensor},
Math. Zeit., 172  (1980), 273-280.

\bibitem{De82} A. Derdzi\'{n}ski, {\em On compact Riemannian manifolds with harmonic curvature},
Math. Ann. 259 (1982), no. 2, 145–152.

\bibitem{Ejiri} N. Ejiri, {\em A negative answer to a conjecture of conformal transformations of Riemannian manifolds}, J. Math. Soc. Japan, 33 (1981), no. 2, 261–266.

 \bibitem{FM} A. Fischer and J. Marsden, {\em Deformations of the scalar curvature},
  Duke Math. J. 42 (1975), no. 3, 519–547.
 
 \bibitem{I} W. Israel, {\em Event horizons in static vacuum space-times },
Phys. Rev., 164 (1967), no. 5, 1776–1779. 

\bibitem{La} J. Lafontaine,  {\em Sur la ge'ometrie d'une generalisation de l'equation differentielle d'Obata}, 
J. Math. Pures Appl., 62 (1983) no. 1, 63-72.

\bibitem{L} F.J. Li, {\em  Rigidity of complete gradient steady solitons with harmonic Weyl tensor}, preprint, (2020),
arXiv.2101.12681.

\bibitem{Kim} J. Kim, {\em On a classification of 4-d gradient Ricci solitons with harmonic Weyl curvature}
J. Geom. Anal., 27 (2017), no. 2, 986-1012. 

\bibitem{KS} J. Kim and J. Shin,  {\em  Four-dimensional static and related critical spaces with harmonic curvature}, 
Pacific J. Math., 295 (2018), no. 2, 429-462.

\bibitem{Ko} O. Kobayashi, {\em  A differential equation arising from scalar curvature fucntion},
J. Math. Soc. of Japan, 34 (1982), no. 4, 665-675.

\bibitem{Kobayashi-Obata} O. Kobayashi and M. Obata, {\em  Conformally-flatness and static space-time. Manifolds and Lie groups}, (Notre Dame, Ind., 1980), pp. 197–206, Progr. Math., 14, Birkh\"auser, Boston, Mass., 1981.

\bibitem{Ob} M. Obata,  {\em Certain conditions for a riemannian manifold to be isometric with a sphere},
J. Math. Soc. Japan, 14 (1962), 333–340.

\bibitem{QY} J. Qing and W. Yuan,  {\em A note on vacuum static spaces and related problems}, 
J. Geom. Phys. 74 (2013), 18–27. 

\bibitem{S} Y. Shen,  {\em A note on Fischer-Marsden's conjecture}, Proc. Amer. Math. Soc., 125 (1997), no. 3, 901–905.

\bibitem{Wald84} R.M. Wald, {\em General Relativity}, Chicago: U. Chicago Press, 1984

\bibitem{YN} K. Yano and T. Nagano,  {\em Einstein spaces admitting a one-parameter group of conformal transformations}, Ann. of Math., 69 (1959), no. 2, 451-461.
\end{thebibliography}
\end{document}